\documentclass{amsart}
\usepackage[utf8]{inputenc}
\usepackage[margin=1in]{geometry}  
\usepackage{graphicx}              
\usepackage{amsmath}               
\usepackage{amsfonts}              
\usepackage{amsthm}                
\usepackage{amssymb}
\usepackage{shuffle}                 
\usepackage{tikz}  
\usetikzlibrary{shapes}
\usetikzlibrary{positioning}
\usepackage{ytableau}    
\usepackage{xcolor}
\usepackage{float}
\usepackage{pifont}

\newtheorem{thm}{Theorem}[section]
\newtheorem{lem}[thm]{Lemma}
\newtheorem{prop}[thm]{Proposition}
\newtheorem{cor}[thm]{Corollary}

\newtheorem{defn}[thm]{Definition}
\newtheorem{rmk}[thm]{Remark}

\numberwithin{thm}{section}

\newcommand{\Des}{\mathrm{Des}}
\newcommand{\Peak}{\mathrm{Peak}}
\newcommand{\Spike}{\mathrm{Spike}}
\newcommand{\Valley}{\mathrm{Valley}}

\newcommand{\mk}{\mathrm{Mark}}
\newcommand{\flip}{\mathrm{fl}}

\newcommand{\Sy}{\mathfrak{S}}
\newcommand{\By}{\mathfrak{B}}
\newcommand{\init}[2]{#2|_{#1}}
\newcommand{\cmark}{\color{green!50!black}\ding{51}}%
\newcommand{\xmark}{\color{red!80!black}\ding{55}}

\title{Connecting descent and peak polynomials}
\author[Kantarci Oguz]{Ezgi Kantarci O\u{g}uz}
\address{Department of Mathematics, University of Southern California, 3620 South Vermont Avenue, Los Angeles, CA 90089-2532, U.S.A.}
\email{kantarci@usc.edu}

\begin{document}

\begin{abstract}
   A permutation $\sigma=\sigma_1 \sigma_2 \cdots \sigma_n$ has a descent at $i$ if $\sigma_i>\sigma_{i+1}$. A descent $i$ is called a peak if $i>1$ and $i-1$ is not a descent. The size of the set of all permutations of $n$ with a given descent set is a polynomials in $n$, called the descent polynomial. Similarly, the size of the set of all permutations of $n$ with a given peak set, adjusted by a power of $2$ gives a polynomial in $n$, called the peak polynomial. In this work we give a unitary expansion of descent polynomials in terms of peak polynomials. Then we use this expansion to give a combinatorial interpretation of the coefficients of the peak polynomial in a binomial basis, thus giving a new proof of the peak polynomial positivity conjecture. 
\end{abstract}
\ytableausetup{notabloids}

\maketitle
\section{Introduction}

Denote by $\Sy_n$ the symmetric group of permutations $\sigma=\sigma_1 \sigma_2 \cdots \sigma_n$ of $[n]=\{1,2,\ldots,n\}$ written in one-line notation. We will draw the graph of $\sigma$ by plotting points $(i,\sigma_i)$ and connecting consecutive points. 

We define the descent set of $\sigma$ as follows:
$$\Des(\sigma)=\{i \mid \sigma_i>\sigma_{i+1}\}\subset [n-1].$$

Note that the descents mark the beginnings of the intervals where the graph is decreasing, as seen in Figure \ref{fig:des} below.

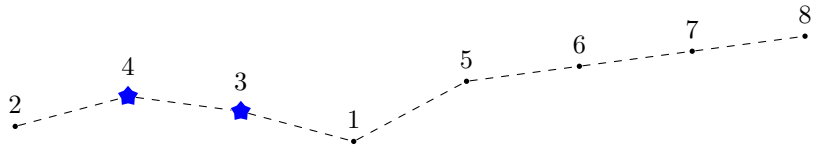
\begin{figure}[h] \label{fig:des}
\begin{tikzpicture}[node distance=.1mm and .1mm,
reg/.style={circle,inner sep=0pt,minimum size=2pt,fill=black},
des/.style={star,star points=5,inner sep=0pt,minimum size=8pt, fill=blue},
peak/.style={regular polygon,regular polygon sides=3,inner sep=0pt,minimum size=6pt, fill=blue}
valley/.style={regular polygon,regular polygon sides=3,inner sep=0pt,minimum size=6pt, fill=green, rotate=180}]

\node(1) [reg] at (1,.4) {}; 
\node [above=of 1] {2};
\node(2) [des] at (2.5,.8) {};
\node [above=of 2] {4};
\node(3) [des] at (4,.6) {}; 
\node [above=of 3] {3};
\node(4) [reg] at (5.5,.2) {};
\node [above=of 4] {1};
\node(5) [reg] at (7,1) {}; 
\node [above=of 5] {5};
\node(6) [reg] at (8.5,1.2) {};
\node [above=of 6] {6};
\node(7) [reg] at (10,1.4) {}; 
\node [above=of 7] {7};
\node(8) [reg] at (11.5,1.6) {};
\node [above=of 8] {8};
\draw[dashed] (1)--(2)--(3)--(4)--(5)--(6)--(7)--(8);
\end{tikzpicture}
\caption{The graph of $\sigma=24315678$ with descents marked in blue.}
\end{figure}

For a given set $S$ and $n>\mathrm{max}(S)$, we let $D(S,n)$ be the set of all permutations in $\Sy_n$ with descent set $S$, and put $d(S,n)=|D(S,n)|$.  In 1915, it was shown by MacMahon \cite{macm} that this is a polynomial in $n$. More recently, Diaz-Lopez et al. \cite{descent} proved this polynomial expands into the binomial basis around $n-m$, where $m=\mathrm{max}(S)$, and gave a combinatorial interpretation for the coefficients. Using the notation $\init{i}{\sigma}= \{\sigma_1,\sigma_2, \ldots, \sigma_i\}$, we give a version of their result slightly altered to include the case $m>\mathrm{max}(S)$ as follows:
\begin{thm}[\cite{descent}] \label{thm:des}
For any finite set of positive integers $S$ with $\mathrm{max}(S)\leq m$ we have:
\begin{equation}
   \displaystyle d(S,n)=a_0(S){n-m\choose0}+a_1(S){n-m\choose1}+\cdots+a_m(S){n-m\choose m},
\end{equation}
where the constant $a_k(S)$ is the number of $\sigma \in D(S,2m)$ such that:
\begin{equation*}
   \init{m}{\sigma} \cap [m + 1, 2m] = [m + 1, m + k].
\end{equation*}
\end{thm}
The underlying idea is quite elegant and will be useful when we are proving a similar result for peak polynomials. Simply put, if $\sigma \in D(S,n)$ has $\{\sigma_1,\sigma_2, \ldots, \sigma_m\} \cap [m + 1, n]=A$ for some $k$ element set $A$, and $B$ is any other $k$ element subset of $[m + 1, n]$, then exchanging elements of $A$ and $B$ while preserving the orders gives another permutation with descent set $S$. Therefore, for each $k$, it is enough to count for the simplest $k$-element subset of $[m+1,n]$ and multiply with ${n-m}\choose{k}$.

For example, there are $3$ elements in $D(\{2,3\},n)$ satisfying   $ \init{m}{\sigma} \cap [5, 8] = \varnothing$ : $14325678, 24315678,34215678$. So we have $a_0(\{2,3\})=3$ for $m=4$. Calculating the other coefficients similarly, we obtain:
\begin{equation}\label{eq:des23}
    \displaystyle d(\{2,3\},n)=3{n-4\choose0}+8{n-4\choose1}+7{n-4\choose2}+2{n-4\choose 3}+0{n-4\choose 4}.   
\end{equation}

Another well studied permutation statistic is given by peak point. Here we define the peak points and their counterpart valley points of the partition to be the points higher and lower than their neighbors respectively:  

$$\Peak(\sigma)=\{\sigma\mid \sigma_i>\sigma_{i+1}, \sigma_{i-1}\}\subset[n-1]/\{1\},$$
$$\Valley(\sigma)=\{\sigma\mid \sigma_i<\sigma_{i+1}, \sigma_{i-1}\}\subset[n-1]/\{1\}.$$

The example $\sigma=34215678$ from Figure \ref{fig:des} has $\Peak(\sigma)=\{2\}$ and $\Valley(\sigma)=\{4\}$. We also set $\Spike(\sigma)=\Peak(\sigma)\cup \Valley(\sigma)$ to be the set of all extremal points that are not corner points.

For a given set $I$ and $n>\mathrm{max}(I)$, we let $P(I,n)$ be the set of permutations in $\Sy_n$ with peak set $I$, and set $p(I,n)=2^{-n+|I|+1}|P(I,n)|$. Note that peaks are more restrictive in the sense that $p(I,n)=0$ if $I$ contains $1$ or any consecutive entries. For the rest of this work, we will focus our attention to \emph{admissible} peak sets $I$: $I\subset [n-1]/\{1\}$ such that $i\in I \Rightarrow i+1 \notin I$. 

In \cite{saganpeak} Billey, Burzdy and Sagan proved that $p(I,n)$ is a polynomial in $n$, and conjectured that the coefficients of this polynomial in a binomial basis centered at $\mathrm{max}(I)$ are non-negative. Their conjecture was proved in 2017 by Diaz-Lopez et al. \cite{omar} using the recursion of peaks, without describing the actual coefficients.

In this work, we tie the theory of peak and descent polynomials together by giving a binary expansion of $d(S,n)$ in terms of peak polynomials. We use this expansion to give a description of the peak polynomial coefficients analogous to the one in Theorem \ref{thm:des}. In Section \ref{sec:marked}, we extend our notions of descents and peaks to $\By_n$, the set of marked permutations of $n$ with $2^n n!$ elements. The added exponent of $2$ cancels out with the $2^{-n+|I|+1}$ from the peak polynomial definition, giving us a way to expand descent polynomials in terms of peak polynomials. In Section $\ref{sec:peak}$, we define involutions on permutations that flip the descents on an initial section and we use them to partition permutations with a given descent set to calculate the coefficients for the peak polynomial.

\section{Descents and Peaks of Marked Permutations}\label{sec:marked}

We start our section by tweaking our notation a little bit to express our formulas easier. Note that the peaks and valleys of a permutation only depend on its descent set. In fact for any $S\in[n-1]$, we can talk about the peaks and valleys of $S$:
$$\Peak(S)=\{1<i\leq n-1 \mid i \in S, {i-1}\notin S\},$$
$$\Valley(S)=\{1<i\leq n-1 \mid i \notin S, {i-1}\in S\}.$$
Note that with this notation, $\Peak(\sigma)=\Peak(\Des(\sigma))$ and $\Valley(\sigma)=\Valley(\Des(\sigma))$ as expected. We also set $\Spike(S):=\Peak(S)\cup \Valley(S)$.

Denote by $\By_n$ the set of signed permutations:
$$\By_n:=\{\rho=\rho_1\rho_2\ldots\rho_n \mid \forall i\leq n \exists k: \rho_k=i \text{ or } \rho_k=-i\}.$$ Note that the definitions of descent, peak, spike and valley naturally extend to signed permutations by saying $i$ is a descent of $\rho$ if $\rho_i>\rho_{i+1}$.

\begin{lem}[\cite{ezgi}] Let $\sigma \in S_n$ have $\Peak(\sigma)=I$. Denote by $\mk(\sigma)$ the $2^n$ element subset of $\By_n$ that give $\sigma$ when marks are erased. Then, for all $\rho \in \mk(\sigma)$, $\Spike(\rho) \supset I$. Conversely, for any $S \subset [n-1]$ satisfying $\Spike(S)\supset I$, there are exactly $2^{|I|+1}$ elements in $\mk(\sigma)$ with descent set equal to $S$.
\end{lem}

\begin{thm} We have $\displaystyle d(S,n)= \sum_{I\subset \Spike(S)} p(I,n)$. \label{thm:dtop}
\end{thm}
\begin{proof} As any marking of the images of $1,2,\ldots,n$ is essentially just a reordering of $[n]$, the number of elements of $\By_n$ with a given descent set $S$ is simply $2^n d(S,n)$. Also note that for all $\rho\in \By_n$ with descent set $S$, $\rho$ is in $\mk(\sigma)$ for some $\sigma$ with $P(\sigma)\subset D(I,n)$. \begin{eqnarray*}2^n d(S,n)&=&\sum_{I\subset \Spike(S)} 2^{|I|+1}|\Peak(I,n)|=\sum_{I\subset \Spike(S)} 2^n p(I,n).
\end{eqnarray*}
\end{proof}

For example:
\begin{itemize}
    \item $d(\varnothing,n)=p(\varnothing,n)=1$.
    \item $d(\{1\},n)=p(\{2\},n)+p(\varnothing,n)$. 
    \item For $1<k<n$, $d(\{k\},n)=p(\{k\},n)+p(\{k+1\},n)+p(\varnothing,n)$.
    \item $d(\{k,k+1,\ldots,k+j\},n)=p(\{k,k+j+1\},n)+p(\{k\},n)+p(\{k+j+1\},n)+p(\varnothing,n)$.
\end{itemize}

For any set $I\in [n]/\{1\}$, we will let $S_I$ denote the unique subset of $[n]/\{n\}$ satisfying $\Spike(S_I)=I$, constructed by alternating the elements of $I$ to be peaks and valleys such that the rightmost one is not a peak. For example, for $I=\{2,4\}$ we have $S_I=\{2,3\}$: the descent set with a peak at $2$ and a valley at $4$ (Figure \ref{fig:si}.

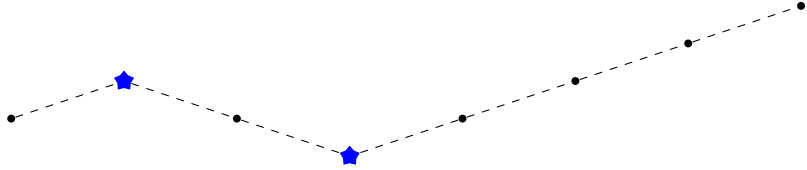
\begin{figure}[h] \label{fig:si}
\begin{tikzpicture}[node distance=.1mm and .1mm,
reg/.style={circle,inner sep=0pt,minimum size=3pt,fill=black},
des/.style={star,star points=5,inner sep=0pt,minimum size=8pt, fill=blue},
peak/.style={regular polygon,regular polygon sides=3,inner sep=0pt,minimum size=6pt, fill=blue}
valley/.style={regular polygon,regular polygon sides=3,inner sep=0pt,minimum size=6pt, fill=green, rotate=180}]

\node(1) [reg] at (1,2/2) {}; 
\node(2) [des] at (2.5,3/2) {};
\node(3) [reg] at (4,2/2) {}; 
\node(4) [des] at (5.5,1/2) {};
\node(5) [reg] at (7,2/2) {}; 
\node(6) [reg] at (8.5,3/2) {};
\node(7) [reg] at (10,4/2) {}; 
\node(8) [reg] at (11.5,5/2) {};
\draw[dashed] (1)--(2)--(3)--(4)--(5)--(6)--(7)--(8);
\end{tikzpicture}
\caption{For $I=\{2,4\}$, $S_I=\{2,3\}$ as seen above.}
\end{figure}

\begin{cor}\label{cor:um}For any admissible set $I$, $\displaystyle p(I,n)=\sum_{J\subset I} (-1)^{|I|-|J|}d(S_J,n).$
\end{cor}
\begin{proof} The formula holds for $I=\varnothing$ as $S_\varnothing=\varnothing$ and $p(\varnothing,n)=d(\varnothing,n)=1$ for all $n$. Assume inductively that the formula holds for all admissible sets with less than $k$ elements. Let $|I|=k$. By Theorem $\ref{thm:dtop}$,
\begin{eqnarray*}
 d(S_I,n)&=& \sum_{J\subset I} p(I,n)= p(I,n)-\sum_{{J\subsetneq I }} \sum_{H\subset J}(-1)^{|J|-|H|}d(S_H, n),\\
 p(I,n)&=&d(S_I,n)- \sum_{{H\subsetneq I }} d(S_H, n) \sum_{t\leq {|I|-|H|-1}} {{|I|-|H|}\choose{t}} {(-1)}^t= d(S_I,n)- \sum_{H\subsetneq I } d(S_H, n) (-(-1)^{|I|-|H|}).
\end{eqnarray*} \end{proof}

If we consider our running example $I=\{2,4\}$, we get the following formulas from Theorem \ref{thm:dtop} and Corollary \ref{cor:um} respectively:
\begin{eqnarray}
d(\{2,3\},n)&=& p(\{2,4\},n)+p(\{2\},n)+p(\{4\},n)+p(\varnothing,n), \\
p(\{2,4\},n)&=&d(\{2,3\},n)-d(\{1\},n)-d(\{1,2,3\},n)+d(\varnothing,n).
\end{eqnarray}

\section{A combinatorial expression for peak coefficients}\label{sec:peak}

We start with defining an operation on permutations that 'flips' the orders of some initial coordinates.

\begin{defn} Let $\sigma\in S_n$. Let $i\leq n$, and $\init{i}{\sigma}=\{a_1<a_2<\cdots<a_i\}$. We define the involution $\flip_i$ as follows:
$$\flip_i(\sigma)_j= \begin{cases} 
      a_{i-k+1} & j\leq i, \sigma_j=a_k \\
      \sigma_{j} & j>i.\\
   \end{cases}
$$
\end{defn}
\begin{figure}[h] 
\begin{tabular}{c c c}
 \begin{tabular}{|c|}  \hline 
\begin{tikzpicture}[node distance=.1mm and .1mm,
reg/.style={circle,inner sep=0pt,minimum size=2pt,fill=black},
des/.style={star,star points=5,inner sep=0pt,minimum size=8pt, fill=red},
peak/.style={star,star points=5,inner sep=0pt,minimum size=8pt, fill=blue},
valley/.style={regular polygon,regular polygon sides=3,inner sep=0pt,minimum size=6pt, fill=green, rotate=180}]
\node(1) [reg] at (.8,.4) {}; 
\node [above=of 1] {2};
\node(2) [peak] at (1.6,.8) {};
\node [above=of 2] {4};
\node(3) [reg] at (2.7,.6) {}; 
\node [above=of 3] {3};
\node(4) [des] at (3.6,.2) {};
\node [above=of 4] {1};
\node(5) [reg] at (4.5,1) {}; 
\node [above=of 5] {5};
\node(6) [reg] at (5.6,1.2) {};
\node [above=of 6] {6};
\node(7) [reg] at (6.3,1.4) {}; 
\node [above=of 7] {7};
\node(8) [reg] at (7.2,1.6) {};
\node [above=of 8] {8};
\draw[dashed] (1)--(2)--(3)--(4)--(5)--(6)--(7)--(8);
\end{tikzpicture}\\
\hline
\end{tabular}
     &\raisebox{.2cm}{\begin{tabular}{c}
       \color{blue} $\flip_2$\\
      \color{blue} 
    $ \mathbf{ \longrightarrow}$  
          \end{tabular} }&
          
 \begin{tabular}{|c|}  \hline 
\begin{tikzpicture}[node distance=.1mm and .1mm,
reg/.style={circle,inner sep=0pt,minimum size=2pt,fill=black},
des/.style={star,star points=5,inner sep=0pt,minimum size=8pt, fill=red},
peak/.style={star,star points=5,inner sep=0pt,minimum size=8pt, fill=blue},
valley/.style={regular polygon,regular polygon sides=3,inner sep=0pt,minimum size=6pt, fill=green, rotate=180}]
\node(1) [reg] at (.8,.8) {}; 
\node [above=of 1] {4};
\node(2) [peak] at (1.6,.4) {};
\node [above=of 2] {2};
\node(3) [reg] at (2.7,.6) {}; 
\node [above=of 3] {3};
\node(4) [des] at (3.6,.2) {};
\node [above=of 4] {1};
\node(5) [reg] at (4.5,1) {}; 
\node [above=of 5] {5};
\node(6) [reg] at (5.6,1.2) {};
\node [above=of 6] {6};
\node(7) [reg] at (6.3,1.4) {}; 
\node [above=of 7] {7};
\node(8) [reg] at (7.2,1.6) {};
\node [above=of 8] {8};
\draw[dashed] (1)--(2)--(3)--(4)--(5)--(6)--(7)--(8);
\end{tikzpicture}\\
\hline
\end{tabular} \\
   \begin{tabular}{c c}
   &\\
    \color{red}$\big\downarrow$ & \color{red} $\flip_4$\\
    &
   \end{tabular}  & &    \begin{tabular}{c c}
   &\\
    \color{red}$\big\downarrow$ & \color{red} $\flip_4$\\
    &
   \end{tabular}\\
         
 \begin{tabular}{|c|}  \hline 
\begin{tikzpicture}[node distance=.1mm and .1mm,
reg/.style={circle,inner sep=0pt,minimum size=2pt,fill=black},
des/.style={star,star points=5,inner sep=0pt,minimum size=8pt, fill=red},
peak/.style={star,star points=5,inner sep=0pt,minimum size=8pt, fill=blue},
valley/.style={regular polygon,regular polygon sides=3,inner sep=0pt,minimum size=6pt, fill=green, rotate=180}]
\node(1) [reg] at (.8,.6) {}; 
\node [above=of 1] {3};
\node(2) [peak] at (1.6,.2) {};
\node [above=of 2] {1};
\node(3) [reg] at (2.7,.4) {}; 
\node [above=of 3] {2};
\node(4) [des] at (3.6,.8) {};
\node [above=of 4] {4};
\node(5) [reg] at (4.5,1) {}; 
\node [above=of 5] {5};
\node(6) [reg] at (5.6,1.2) {};
\node [above=of 6] {6};
\node(7) [reg] at (6.3,1.4) {}; 
\node [above=of 7] {7};
\node(8) [reg] at (7.2,1.6) {};
\node [above=of 8] {8};
\draw[dashed] (1)--(2)--(3)--(4)--(5)--(6)--(7)--(8);
\end{tikzpicture}\\
\hline
\end{tabular}
     &\raisebox{.2cm}{\begin{tabular}{c}
       \color{blue} $\flip_2$\\
      \color{blue} 
    $ \mathbf{ \longrightarrow}$  
          \end{tabular} }&
          
 \begin{tabular}{|c|}  \hline 
\begin{tikzpicture}[node distance=.1mm and .1mm,
reg/.style={circle,inner sep=0pt,minimum size=2pt,fill=black},
des/.style={star,star points=5,inner sep=0pt,minimum size=8pt, fill=red},
peak/.style={star,star points=5,inner sep=0pt,minimum size=8pt, fill=blue},
valley/.style={regular polygon,regular polygon sides=3,inner sep=0pt,minimum size=6pt, fill=green, rotate=180}]
\node(1) [reg] at (.8,.2) {}; 
\node [above=of 1] {1};
\node(2) [peak] at (1.6,.6) {};
\node [above=of 2] {3};
\node(3) [reg] at (2.7,.4) {}; 
\node [above=of 3] {2};
\node(4) [des] at (3.6,.8) {};
\node [above=of 4] {4};
\node(5) [reg] at (4.5,1) {}; 
\node [above=of 5] {5};
\node(6) [reg] at (5.6,1.2) {};
\node [above=of 6] {6};
\node(7) [reg] at (6.3,1.4) {}; 
\node [above=of 7] {7};
\node(8) [reg] at (7.2,1.6) {};
\node [above=of 8] {8};
\draw[dashed] (1)--(2)--(3)--(4)--(5)--(6)--(7)--(8);
\end{tikzpicture}\\
\hline
\end{tabular}
\end{tabular}
\caption{Operations $\flip_2$ and $\flip_4$ on $\sigma=24315678$. \label{fig:flip}}
\end{figure}
See Figure \ref{fig:flip} for examples.  
\begin{rmk}
The involution $\flip_i$ satisfies the following:
\begin{itemize}
    \item $\init{i}{\flip_i(\sigma)}=\init{i}{\sigma}$.
    \item For $k<i$, $k$ is a descent of $\flip_i(\sigma)$ iff it is not a descent of $\sigma$. 
    \item For $k>i$, $k$ is a descent of $\flip_i(\sigma)$ iff it is a descent of $\sigma$. 
    \item $\flip_i$ exchanges all the peaks less than $i$ with valleys, and all the valleys less than $i$ with peaks.
\end{itemize}\end{rmk}
 In Figure \ref{fig:flip}, we see an instance of operations $\flip_2$ and $\flip_4$ commuting. Now we will prove that this is the case in general.
\begin{prop} For all $i,j$, $\flip_i$ and $\flip_j$ commute.\label{prop:com}
\end{prop}
\begin{proof} Assume without loss of generality that $\init{j}{\sigma}=[j]$, so that $\flip_j(\sigma)_k=j-\sigma_k+1$. 

Put $\init{i}{\sigma}=\{a_1<a_2<\cdots<a_i\}$. 
$$\flip_i(\sigma)_j= \begin{cases} 
      a_{i-k+1} & if\sigma_j=a_t \text{ for some }t \\
      \sigma_{j} & \text{otherwise}.\\
   \end{cases}
$$
As $\init{i}{\flip_j(\sigma)}=\{j-a_i+1<j-a_{i-1}+1<\ldots<j-a_1+1\}$, we have
\begin{eqnarray*}
    (\flip_i \circ \flip_j (\sigma))_k&=&
    \begin{cases} j-a_{i-t+1}+1 & \text{ if } j- \sigma_k+1=j-a_t+1 \text{ for some $t$}\\
    j-\sigma_k+1  & \text{ otherwise} \end{cases} \\
    &=& \begin{cases}  j-a_{i-t+1}+1 & \text{ if } \sigma_k=a_t \text{ for some $t$}\\
    j-\sigma_k+1  & \text{ otherwise} \end{cases} \\
    &=&( \flip_j \circ \flip_i (\sigma))_k.
\end{eqnarray*}
\end{proof}

\begin{defn} For $i \in \Spike(\sigma)$, $\sigma$ is said to admit an $i^+$-flip if $\Spike(\flip_i(\sigma))=\Spike(\sigma)/\{i\}$. Similarly, it is said to admit an $i^-$-flip if $\Spike(\flip_{i-1}(\sigma))=\Spike(\sigma)/\{i\}$. We say $\sigma$ is admits an $i$-flip if it admits an $i^+$- or $i^-$- flip. For $\sigma$ that admits and $i$-flip, we set
\begin{equation*}
\Psi_i(\sigma)= 
\begin{cases} \flip_i(\sigma) & \sigma \text{ admits an $i^+$-flip,}\\
\flip_{i-1}(\sigma) & \text{ otherwise.}
\end{cases}
\end{equation*}
\end{defn}

Visually, this means that $\flip_i$ or $\flip_{i-1}$ straightens out the peak or valley point at $i$. For example, $24315678$ from Figure \ref{fig:flip} admits a $4$-flip, but not a $2$-flip.

\begin{prop}For all $i,j$ such that $|i-j|>1$, $\sigma$ admits an $i^+$-flip if and only if $\Psi_j(\sigma)$ admits an $i^+$-flip. Similarly, $\sigma$ admits an $i^-$-flip if and only if $\Psi_j(\sigma)$ admits an $i^-$-flip.
\end{prop}
\begin{proof} Assume $j<i$. Note that whether a permutation $\sigma$ with a spike at $i$ admits an $i^+$ or $i^-$-flip only depends on the images of $i-1,i$ and $i+1$ in $\flip_i(\sigma)$ and $\flip_{i-1}(\sigma)$. As $\flip_j$ and $\flip_{j-1}$ do not alter these images, $\Psi_j$ does not change whether $\sigma$ admits an $i^+$ or $i^-$ flip. The case $j<i$ follows by Proposition \ref{prop:com}.
\end{proof}

For any admissible set $I=\{i_1,i_2,\ldots,i_k\}$, we put  $\Psi_{I}:=\Psi_{i_1}\circ\Psi_{i_2}\circ\cdots \circ\Psi_{i_k}$. This operation is well-defined by Proposition \ref{prop:com}, and the corollary above.

\begin{lem}\label{lem:flip} For any admissible set $I$ and any $J\subset I$, $\Psi_{J}$ induces a bijection between elements of $D(S_I,n)$ that admit a $j$-flip for all $j \in J$ and elements of $D(S_{I/J},n)$. In particular, for any $m\geq \mathrm{max}(I)$ we have:
\begin{equation*}
   \init{m}{\sigma} \cap [m + 1, 2m] = [m + 1, m + k] \Longleftrightarrow    \init{m}{\Psi_{J}(\sigma)} \cap [m + 1, 2m] = [m + 1, m + k].
\end{equation*}
\end{lem}
\begin{proof} Assume $\sigma \in D(S_I,n)$ admits a $j$-flip for each $j\in J$. Then $\Psi_{J}(\sigma)$ has spike set $I/J$ by definition. As $\mathrm{max}(\Des( \Psi_{J}(\sigma)))\leq \mathrm{max}(I/J)$, its descent set is the set $S_{I/J}$. 

For the converse, let $\rho \in D(S_{I/J},n)$, and $j\in J$. If $\flip_j(\rho)$ has a spike point at $j$, then $\flip_j(\rho) \in D(S_{I/J \cup \{j\}},n)$. If, on the other hand $\flip_j(\rho)$ has a spike point at $j$, note that as $I$ is an admissible set, neither $j-1$ nor $j+1$ are spike points, so $j,j-1$ and $j-2$ are either all descents or all not descents. Assume without generality that all are descents, the other case being symmetrical. Then, as $\flip_j(\rho)$ has no spike at $j$, we have $(\flip_j(\rho))_{j-2}<(\flip_j(\rho))_{j-1}<(\flip_j(\rho))_{j}<\rho_{j+1}<\rho_{j}$. As $(\flip_j(\rho))_{j-1}<\rho_j$, $(\flip_{j-1}(\rho))_{j-2}=(\flip_{j-1}(\rho)_{j-1})<(\flip{j}(\rho))_{j-1}<\rho_{j}$, $\flip_{j-1}(\rho)$ has spikes at $I/J \cup {j}$. It also admits a $j^-$-flip but not a $j^+$ flip. 

Doing $\flip_j$ if $\flip_j(\rho)$ has a spike point at $j$, and $\flip_{j-1}$ otherwise gives an element $\sigma^{(j)}$ of $D(S_{I/J\cup\{j\}},n)$ with $\Psi_j(\sigma^{(j)})=\rho$. So, by Lemma \ref{lem:flip}, we have an element $\sigma \in D(S_I,n)$ with $\Psi_J(\sigma))=\rho$. The second part follows as $\sigma|_m=\flip_i(\sigma)|_m$ for any $i\leq n$.

\end{proof}

\begin{thm} For any admissible set of $I$ with $\mathrm{max}(I)\leq m$ we have
\begin{equation}
   \displaystyle p(I,n)=b_0(I){n-m\choose0}+b_1(I){n-m\choose1}+\cdots+b_m(I){n-m\choose m},
\end{equation}
where the constant $b_k(I)$ is the number of $\sigma \in D(S_I,2m)$ such that:
\begin{equation*}
   \init{m}{\sigma} \cap [m + 1, 2m] = [m + 1, m + k],
\end{equation*}  and $\sigma$ does not admit any $i$-flips.
\end{thm}
\begin{proof} Let $I$ be an admissible set with max $m$.  Note that for any $J\subset I$, $\mathrm{max}(S_J)\leq \mathrm{max}(J)-1\leq m$. Fix $k\leq n$.

For any $J\subset I$, we use the notation $B_{I/J}$ to denote the set of $\sigma \in D(S,2m)$ such that:
$ \init{m}{\sigma} \cap [m + 1, 2m] = [m + 1, m + k]$ and $\sigma$ admits $j$ flips for all $j \in J$. 

Recall from Corollary \ref{cor:um} that we have: 
\begin{eqnarray}
p(I,n)=\sum_{J\subset I} (-1)^{|I|-|J|}d(S_J,n).
\end{eqnarray}
Combining this with Theorem \ref{thm:des} we get:
\begin{eqnarray*}
a_k(I)&=& \sum_{J\subset I} (-1)^{|I|-|J|}b_k(J)\\
&=& \sum_{J\subset I} (-1)^{|I|-|J|} |B_{J}|\\
&=&\displaystyle  b_k(I) -|\cup_J B_J|
\end{eqnarray*}
by Lemma \ref{lem:flip} and the inclusion-exclusion principle.
\end{proof}
\begin{table}[ht]
    \centering
\begin{tabular}{c c c c}
\begin{tabular}{| c|c c|}
\hline
 & $2$-flip & $4$-flip\\
\hline
14325678 &\xmark&\cmark   \\
24315678&\xmark&\cmark\\
34215678&\cmark&\cmark\\
&&\\
&&\\
&&\\
&&\\
&&\\
\hline
\end{tabular}     & 
\begin{tabular}{| c|c c|}
\hline
 & $2$-flip & $4$-flip\\
\hline
15324678 &\xmark&\cmark   \\
 15423678&\xmark&\xmark\\
 25314678&\xmark&\xmark\\
25413678&\xmark&\xmark\\
 35214678&\cmark&\xmark\\
 35412678&\xmark&\xmark\\
45213678&\cmark&\xmark\\
45312678&\cmark&\xmark\\
\hline
\end{tabular}&
\begin{tabular}{|c|c c|}
\hline
 & $2$-flip & $4$-flip\\
\hline
  16523478 &\xmark&\xmark   \\
 26513478&\xmark&\xmark\\
 36512478&\xmark&\xmark\\
46512378&\xmark&\xmark\\
 56213478&\cmark&\xmark\\
 56312478&\cmark&\xmark\\
56412378&\cmark&\xmark\\
&&\\
\hline
\end{tabular}&
\begin{tabular}{| c|c c|}
\hline
 & $2$-flip & $4$-flip\\
\hline
      57612348&\xmark& \xmark  \\
675123478&\cmark&\xmark\\
&&\\
&&\\
 &&\\
&&\\
&&\\
&&\\
\hline
\end{tabular}\\
\centering $k=0$ & $k=1$& $k=2$&  $k=3$\\
&&&
\end{tabular}

    \caption{The elements $\sigma \in D({\{2,3\}},8)$ satisfying    $\init{4}{\sigma} \cap [5, 8] = [5, 4 + k]$. \label{tab:1}}
    
\end{table}

We will end this section by calculating the expansion of $p(\{2,4\},n)$. Recall that $S_{\{2,4\}}=\{2,3\}$. For all elements $\sigma \in  D({\{2,3\}},8)$ satisfying  $\init{4}{\sigma} \cap [5, 8] = [5, 4+ k]$ for some $k$ we need to check if $\sigma$ admits a $2$-flip or a $4$-flip. Checking for $2$-flips is very straightforward, we just need to check whether $\sigma_1>\sigma_3$. $4$-flips are slightly more tricky as $\flip_4$ does not simply exchange a pair of coordinates, and we actually need to calculate $\flip_4(\sigma)$ to see if $\flip_4(\sigma)_3$ is smaller than $\sigma_5$. For each $k$, the related permutations $\sigma$ can be found in Table \ref{tab:1}, along with the information on whether they admit $2$ or $4$-flips.

Counting the elements that admit neither $2$ nor $4$-flips from Table \ref{tab:1} gives us the following formula:
\begin{equation*}
  p(\{2,4\},n)=0{n-4\choose0}+4{n-4\choose1}+4{n-4\choose2}+1{n-4\choose 3}.    
\end{equation*}
In fact, the inclusion-exclusion principle allows us to read the coefficients for $p({2,n})$(ones that admit only $4$-flips), $p({4,n})$(ones that admit only $2$-flips) and $p(\varnothing,n)$(ones that admit only $4$-flips) from Table $\ref{tab:1}$:
\begin{eqnarray*}
p(\{2\},n)&=&2{n-4\choose0}+1{n-4\choose1}+0{n-4\choose2}+0{n-4\choose 3},   \\
p(\{4\},n)&=&0{n-4\choose0}+3{n-4\choose1}+3{n-4\choose2}+1{n-4\choose 3},   \\
p(\varnothing,n)&=&1{n-4\choose0}+0{n-4\choose1}+0{n-4\choose2}+0{n-4\choose 3}.   
\end{eqnarray*}
Note that $  p(\{2,4\},n)+p(\{2\},n)+p(\{4\},n)+p(\varnothing,n)=d(\{2,3\},n)$ as required.
\section{Acknowledgements}
The author would like to thank Mohamed Omar for an inspiring seminar talk on the subject. The author is also immensely grateful to Alexander Diaz-Lopez and Erik Insko for spotting an error with the initial statement of the main result, and their many helpful suggestions and comments in the following discussion. This work was partially supported by the USC Graduate School Final Year Fellowship.

\bibliographystyle{plain}
\bibliography{main}

\end{document}